\documentclass[11pt,reqno]{amsart}
\usepackage{microtype}
  \usepackage[margin=4.1cm]{geometry}
\numberwithin{equation}{section}
\usepackage{amsmath, amsfonts,amsthm,amssymb,amscd, verbatim,graphicx,color,multirow,booktabs, caption,tikz,tikz-cd, mathdots,bm}
\usepackage{tikz-cd}
 \usepackage[pagebackref]{hyperref} 
\usetikzlibrary{positioning} 

\newtheorem{theorem}{Theorem}[section]
\newtheorem{lemma}[theorem]{Lemma}

\newtheorem{proposition}[theorem]{Proposition}
 \newtheorem{corollary}[theorem]{Corollary}

       \newtheorem*{conjecture}{Conjecture}

      \theoremstyle{definition}
     \newtheorem*{definition}{Definition}
     \newtheorem{example}[theorem]{Example}
     
     \theoremstyle{remark}
     \newtheorem{remark}[theorem]{Remark}

\newcommand{\Sym}{\mathop{\mathrm{Sym}}}
\newcommand{\Alt}{\mathop{\mathrm{Alt}}}
\newcommand{\Aut}{\mathop{\mathrm{Aut}}}

\newcommand{\GL}{\mathop{\mathrm{GL}}}

\newcommand{\SL}{\mathop{\mathrm{SL}}}
\newcommand{\SLL}{\mathop{\mathrm{SL_2}}}

\newcommand{\Sol}{\mathop{\mathrm{Sol}}}
\newcommand{\diam}{\mathop{\mathrm{diam}}}

\newcommand{\Kaz}{\mathop{\mathrm{Kaz}}}
\newcommand{\Kazz}{\mathop{\mathrm{Kaz^*}}}
\newcommand{\Cay}{\mathop{\mathrm{Cay}}}
\newcommand{\gap}{\mathop{\mathrm{gap}}}
\newcommand{\Irr}{\mathop{\mathrm{Irr}}}
\newcommand{\End}{\mathop{\mathrm{End}}} 
\newcommand{\Prob}{\mathop{\mathrm{Prob}}}

 \definecolor{mygold}{rgb}{0.55,0.35,0.0}
 \definecolor{mycolor}{rgb}{0.55,0.0,0.16}
  \definecolor{myred}{rgb}{0.75,0.0,0.16} 
  \definecolor{mygreen}{rgb}{0.0,0.4,0.16} 
  \definecolor{myviolet}{rgb}{1,0,1} 
   \definecolor{mypink}{rgb}{0.9,0,0.5}
	
	\newcommand{\MYhref}[3][black]{\href{#2}{\color{#1}{#3}}}%

\hypersetup{
colorlinks=true,
linkcolor=true,
linktocpage=true,
pageanchor=true,
hyperindex=true
}

 \AtBeginDocument{
     \hypersetup{
  linkcolor=mygold,
  urlcolor=mypink,
citecolor=mygreen
}
     }

\makeatletter
\@namedef{subjclassname@2020}{%
  \textup{2020} Mathematics Subject Classification}
\makeatother


\author[Luca Sabatini]{Luca Sabatini}
\address{\parbox{\linewidth}{Luca Sabatini, Mathematics Institute, University of Warwick\\
Coventry CV4\,7AL, United Kingdom \vspace{0.1cm}}}
\email{luca.sabatini@warwick.ac.uk, sabatini.math@gmail.com} 

\begin{document} 
 \title[Cayley graphs of quasirandom groups]{Cayley graphs of quasirandom groups} 

\maketitle 

\begin{abstract} 
A finite group $G$ is $\varepsilon$-quasirandom
if all its nontrivial irreducible complex representations
have degree at least $|G|^\varepsilon$.
Building on recent work of Golsefidy--Srinivas,
we prove that expansion in a quasirandom group is controlled by
expansion in its simple quotients.
As a consequence, we remove the product theorem from the hypotheses of
the Bourgain--Gamburd expansion machine.
Moreover, we combine this result with crown theory to deduce that $1 + \lfloor \varepsilon^{-1} \rfloor$
random elements give an expander Cayley graph with high probability.
Finally, generalizing results of Breuillard--Green--Tao and Pyber--Szab\'o,
we prove that the diameter of any connected Cayley graph of a quasirandom group is polylogarithmic.
\end{abstract} 


\vspace{0.2cm}
\section{Introduction} 

Expander graphs are highly connected sparse graphs and play a fundamental role
in various areas of mathematics and computer science \cite{HLW06,Lub12}.
Among their remarkable properties,
there is the fact that their diameter is logarithmic in the number of vertices.

Several notable constructions of expanders involve Cayley graphs.
In a spectacular breakthrough, Bourgain and Gamburd \cite{BG08} devised an ingenious method,
starting from the so-called Sarnak--Xue trick \cite{SX91},
to show that $2$ random elements in $\SL_2(p)$ give an expander Cayley graph with high probability.
Using the same strategy,
the companion paper of Breuillard--Green--Guralnick--Tao \cite{BGGT15} generalized their theorem
to all finite simple groups of Lie type of bounded rank. 
There are also weaker results that have the advantage of being valid for literally all generating sets.
In a parallel development,
Helfgott \cite{Hel08} proved that the diameter of every connected Cayley graph of $\SL_2(p)$ is polylogarithmic,
while Breuillard--Green--Tao \cite{BGT11} and Pyber--Szab\'o \cite{PS16} generalized
his theorem to all finite simple groups of Lie type of bounded rank.

The following property of finite simple groups of bounded Lie rank was first isolated by Gowers \cite{Gow08},
and plays a crucial role in the proofs of these results:

 \begin{definition}
Let $\varepsilon>0$.
A finite group $G$ is $\varepsilon$\emph{-quasirandom} if every nontrivial irreducible complex representation of $G$
has degree at least $|G|^\varepsilon$.
\end{definition} 

We are interested in this property when $\varepsilon$ is fixed and $G$ is large,
so we simply refer to \emph{quasirandom group(s)} in this case.
Quasirandom groups have several characterizations \cite{BS25,NP11},
and in the context of expansion and diameter bounds,
they can be thought of as the groups where the ``tricks''
of Sarnak--Xue \cite{SX91} and Gowers \cite{NP11} are available efficiently.
It is natural to ask:
\emph{how far quasirandomness alone drives expansion phenomena?}
The present paper answers this question completely.

Our first main result builds on the recent work of Golsefidy and Srinivas \cite{GS24,GS25}.
They study group extensions
$$ 1 \to N \hookrightarrow G \twoheadrightarrow Q \to 1 $$
where $Q$ is quasirandom.
In particular, they give hypotheses under which expansion in $\Cay(G,S)$
follows from expansion in $\Cay(G/N,S)$
(in fact, they first prove that such hypotheses induce quasirandomness to $G$ itself).
For instance, they require $G$ to be a specific direct product of $N$ and $Q$,
or $N$ to be small and abelian, or nilpotent of bounded class,
together with suitable technical conditions on the conjugation action of $G$ on $N$.
We show that no additional hypothesis is needed when $G$ is quasirandom:
expansion in quasirandom groups is always controlled by expansion in the simple quotients.

\begin{theorem} \label{thProjection}
Let $G$ be $\varepsilon$-quasirandom.
If $G=\langle S \rangle$, then
$$ \gap(G,S) \> \gg_{\varepsilon,|S|} \> 
\min \{ \gap(G/N,S)  : N~\text{a maximal normal subgroup of}~G\} . $$
\end{theorem}

Here $\gap(G,S)$ denotes the spectral gap of the Cayley graph.
It is remarkable that Theorem~\ref{thProjection} fails dramatically for non-quasirandom groups:
a trivial example is a large cyclic $2$-group,
but we refer the reader to Example~\ref{exBad} for a decisive example with a large unique simple quotient.

The proof of Theorem~\ref{thProjection} does not require a delicate analysis of random walks as in \cite{GS24,GS25}.
Instead, we use some of their results as a black box and proceed with a different strategy.
A fundamental observation is that, by Gr\"un's lemma on perfect groups,
we can implement a process where at each step we quotient the ambient group by its center
and then by a non-central minimal abelian normal subgroup.
This process stops when the solvable radical is trivial,
and it lasts at most $\varepsilon^{-1}$ steps by quasirandomness.
To handle the non-central steps we combine a theorem in \cite{GS25} with
the ``affine conjugating trick'' by Eberhard--Murphy--Pyber--Szab\'o \cite{EMPS25}.
The case of central extensions requires a separate argument that relies on Kazhdan constants
(Lemma~\ref{lemKazZ}).
When the solvable radical is trivial, we actually have a direct product of simple groups by a result
of Barbieri and the author \cite{BS25}.
Finally, this situation can be handled by an iterative application of another theorem in \cite{GS25}.\\

Theorem~\ref{thProjection} has some applications that we now describe.
To start, Theorems~C, D and E in \cite{GS25} are not used in our proof
and in fact can be considered special cases.
Notably, it also allows us to remove the product theorem (Stage 2)
from the hypotheses in the Bourgain--Gamburd expansion machine \cite[Theorem~1.4.2]{Tao15}
(see Corollary~\ref{corBGFinal} below).
With some more work, we prove the following definitive version of Bourgain--Gamburd-type results:

\begin{theorem} \label{thExpRan}
If $G$ is $\varepsilon$-quasirandom,
 then $1 + \lfloor \varepsilon^{-1} \rfloor$ random elements give an expander Cayley graph with high probability.
 \end{theorem}
 
The linear dependence on $\varepsilon^{-1}$ is necessary,
simply because an $\varepsilon$-quasirandom group
may require so many elements to be generated (see Example~\ref{exQ}(c)).
In fact, to obtain Theorem~\ref{thExpRan} from Theorem~\ref{thProjection} and \cite{BGGT15},
it is sufficient to find a generating set.
For this purpose,
we use results of Detomi--Lucchini \cite{DL03} and Lucchini--Morini \cite{LM02} from crown theory
to show that $1 + \lfloor \varepsilon^{-1} \rfloor$ random elements generate $G$ with high probability.

A different application of Theorem~\ref{thProjection} concerns uniform expansion.
The following deep conjecture of Breuillard is reported as ``folklore'' in \cite[Conjecture~4.12]{Bre14}:
\emph{If $G$ is a finite simple group of Lie type of bounded rank,
then all connected Cayley graphs of $G$ are uniform expanders.}
Progress on this problem was made by Breuillard and Gamburd \cite{BG10}
(see also Becker--Breuillard \cite{BB26}), but the general conjecture remains open.
By Theorem~\ref{thProjection}, at least for sets of bounded cardinality,
Breuillard's conjecture is equivalent to a statement for all quasirandom groups.
We push this analogy further and propose the following conjecture for all generating sets:

\begin{conjecture} \label{conjUnif}
If $G$ is $\varepsilon$-quasirandom,
then all connected Cayley graphs of $G$ are $\varepsilon$-uniform expanders.
\end{conjecture}

This appears to be a very difficult problem.
Using our methods we are able to prove a weaker, unconditional uniform result concerning diameters.
Let us write $\diam(G)$ for the maximum diameter over all connected Cayley graphs of $G$.

\begin{theorem} \label{thDiam} 
If $G$ is $\varepsilon$-quasirandom,
then $\diam(G) \leq (\log|G|)^{O_\varepsilon(1)}$.
\end{theorem} 

This extends the aforementioned
theorems of Breuillard--Green--Tao \cite{BGT11} and Pyber--Szab\'o \cite{PS16}.
In the recent paper \cite{EMST26},
the author together with S.~Eberhard, E.~Maini and G.~Tracey
proved a general-purpose polylog-type diameter bound for arbitrary finite groups,
but a different argument is required to obtain Theorem~\ref{thDiam}.
A new ingredient
is that every element of an $\varepsilon$-quasirandom group
is the product of at most $\varepsilon^{-1}$ commutators (Lemma~\ref{lemWid}).\\

The paper is organized as follows.
In Section~\ref{secPre} we present the necessary background.
Moreover, we prove that the solvable radical of an $\varepsilon$-quasirandom group
has $\varepsilon$-bounded derived length
(in fact the proof of this result inspired the multi-step process that we have described above).
In Section~\ref{secGen} we show that
$1+\lfloor \varepsilon^{-1} \rfloor$ random elements generate an $\varepsilon$-quasirandom group with high probability.
It is important to do so here, in order to apply the affine conjugating trick efficiently in later stages.
In Section~\ref{secDia} we prove Theorem~\ref{thDiam},
and in Section~\ref{secExp} we deal with central extensions and complete the proofs of
Theorems~\ref{thProjection} and \ref{thExpRan}.

\vspace{0.1cm} \noindent
{\bfseries Notation.}
Given two positive functions $f$ and $g$,
by $f \ll g$ and $f = O(g)$ we mean the same thing,
namely that there exists a constant $C$ such that $f \leq C g$.
The constant $C$ may depend on one or more parameters which we indicate in subscript.

\vspace{0.1cm}
\section{Preliminaries} \label{secPre}

\subsection{Quasirandom groups} 

Every group in this paper is finite.
A finite  group is $\varepsilon$-quasirandom for some $\varepsilon >0$ if and only if it is perfect.
It is known that there is no $0.35$-quasirandom group,
and that the only $\tfrac{1}{3}$-quasirandom groups are two sporadic simple groups \cite[Corollary~D]{BS25}.
In particular, in the relevant cases we can always assume $\varepsilon < 1/3$.
We say that a group (or more precisely a sequence of groups)
is \emph{quasirandom} to mean that it is $\varepsilon$-quasirandom for a fixed $\varepsilon >0$.

It follows from the definition that if $G$ is $\varepsilon$-quasirandom and $N$ is a proper normal subgroup of $G$,
then $|G/N| \geq |G|^\varepsilon$ and $G/N$ is again $\varepsilon$-quasirandom
(this fact will be used without further mention).
We now report an elegant characterization of quasirandom groups with respect to group actions,
which is due to Nikolov and Pyber \cite{NP11}.

\begin{lemma} \label{lemQSet}
Let $G$ be a finite group.
If $G$ is $\varepsilon$-quasirandom and acts non-trivially on a finite set $\Omega$,
then $|\Omega| > |G|^\varepsilon$.

Conversely, if for all nontrivial actions of $G$ on a finite set $\Omega$ we have $|\Omega| \geq |G|^\varepsilon$,
then $G$ is $\varepsilon'$-quasirandom for some $\varepsilon'(\varepsilon) >0$.
\end{lemma}
\begin{proof}
Let $|\Omega|=n$.
We have a nontrivial homomorphism $G \to \Sym(n)$,
and $\Sym(n)$ can be naturally embedded in $\GL_n(\mathbb{C})$.
Removing the diagonal subspace we obtain a nontrivial linear representation of degree $n-1$ and the claim follows.

Conversely, suppose that $G$ has a nontrivial irreducible linear representation of degree $n$.
By \cite[Theorem~4]{NP11} we obtain a nontrivial action of $G$ on a set of cardinality at most $c_0 n^2$
where $c_0$ is an absolute constant.
Therefore $c_0 n^2 \geq |G|^\varepsilon$ and the proof follows.
\end{proof}

We will often use the following very special case.

\begin{corollary} \label{corQChief}
Let $G$ be $\varepsilon$-quasirandom, and $H,N$ normal subgroups of $G$ with $N \subseteq H$.
If $H/N$ is not central in $G/N$,
then $|H:N| > |G|^\varepsilon$.
\end{corollary}
\begin{proof}
Observe that $G/N$, and so $G$, acts non-trivially on $H/N$ by conjugation
(the identity element is fixed), and use Lemma~\ref{lemQSet}.
\end{proof}

Moreover, if $H<G$ and we consider the action on the right cosets by multiplication,
we obtain $|G:H|>|G|^\varepsilon$, i.e. $|H| < |G|^{1-\varepsilon}$.

The \emph{solvable radical} $\Sol(G)$ is the largest solvable normal subgroup of $G$.
The following result from \cite{BS25} gives a structural characterization of quasirandom groups.

\begin{theorem}[Barbieri--Sabatini] \label{thBS}
Fix $\varepsilon>0$ and let $G$ be $\varepsilon$-quasirandom.
If $|G|$ is sufficiently large with respect to $\varepsilon$, then $|\Sol(G)| < |G|^{1-\varepsilon}$,
and $G/\Sol(G)$ is a direct product of $\varepsilon$-boundedly many finite simple groups
of Lie type of $\varepsilon$-bounded rank.
\end{theorem}

Theorem~\ref{thBS} is especially meaningful for groups with trivial solvable radical,
and we refer the reader to Example~\ref{exQ} for some quasirandom groups with large solvable normal subgroups.

\subsection{Crowns and generation}

In this subsection we recall some results from the papers \cite{DVL98,DL03,LM02},
as well as from the convenient survey \cite{Luc23}.

The \emph{Frattini subgroup} of a finite group is the intersection of the maximal subgroups.
The following observation will be useful.

\begin{remark} \label{remZF}
In a finite perfect group $G$ the center is contained in the Frattini subgroup.
In fact, suppose that $Z(G) \nsubseteq H$ for some maximal subgroup $H$ of $G$.
Then $Z(G) H =G$ and so $G = G' = H' \subseteq H$, which is impossible.
\end{remark}

A finite group $L$ is \emph{primitive monolithic} if it has a unique minimal normal subgroup and trivial Frattini subgroup.
For a monolithic primitive group $L$ with socle $A$,
the \emph{crown-based power} of level $k$ is defined by
$$ L_k \> := \> \{ (l_1,\ldots,l_k) \in L^k : l_1 \equiv \cdots \equiv l_k \mod A \} . $$
In particular, $L_k$ is an extension of $A^k$ by $L/A$.

Let $d(G)$ be the minimal cardinality of a generating set.
In \cite{DVL98}, Dalla Volta and Lucchini proved that for every finite group $G$
there exist a primitive monolithic group $L$ and $k$
such that the crown-based power $L_k$ is a quotient of $G$ and $d(G) = d(L_k)$.
We now see how to attach a family of crown-based powers to an arbitrary finite group $G$.
A chief factor $A = H/N$ of $G$ is called \emph{non-Frattini} if $H/N$ is not contained in the Frattini subgroup of $G/N$.
For a non-Frattini chief factor $A$ of $G$, we associate a monolithic primitive group by
$$ L_A(G) \> := \> 
\begin{cases}
  A \rtimes (G/C_G(A)) \hspace{0.5cm} \text{ if } A \text{ is abelian,}  \\
 G/C_G(A) \hspace{1.6cm} \text{ otherwise.}  
\end{cases} $$
We simply write $L_A$ when $G$ is clear from context.
Let $\delta_A(G)$ be the maximum $k$ such that, for $L = L_A(G)$,
$G$ has a quotient isomorphic to the crown-based power $L_k$.

We come to the probabilistic framework.
For a positive integer $t$, let $P_G(t)$ be the probability that $t$ random elements generate $G$.
In the seminal paper \cite{DL03},
Detomi and Lucchini described how to factorize $P_G(t)$ in terms of the crown-based powers associated to $G$.
For a monolithic primitive group $L$ with socle $A$, let
$$ P_{L,A}(t) \> := \> \frac{P_L(t)}{P_{L/A}(t)} , $$
i.e. the probability that $t$ random elements generate $L$ given that they generate $L/A$.
 Moreover, define
$$ \widetilde{P}_{L,1}(t) \> := \> P_{L,A}(t) , \hspace{1cm} 
\widetilde{P}_{L,2}(t) \> := \> P_{L,A}(t) - \frac{\gamma_A}{|A|^t} , $$ 
where $\gamma_A = |C_{\Aut(L)}(L/A)|$, and for $i \geq 3$,
$$ \widetilde{P}_{L,i}(t) \> := \> P_{L,A}(t) - \frac{(1+q_A+\cdots+q_A^{i-2})\gamma_A}{|A|^t} , $$ 
where $q_A = |\End_L(A)|$ if $A$ is abelian and $q_A =1$ otherwise.
The following formula is \cite[Theorem~18]{DL03}.

\begin{theorem}[Detomi--Lucchini] \label{thDL}
We have
$$ P_G(t) \> = \>
\prod_{A} \left( \prod_{1 \leq i \leq \delta_A(G)} \widetilde{P}_{L_A,i}(t) \right) , $$
where $A$ runs in the set of irreducible $G$-groups $G$-equivalent to a non-Frattini
chief factor of $G$,
and $L_A$ is the monolithic primitive group associated with $A$.
\end{theorem}

Finally, the following is \cite[Theorem~1.1]{LM02}:

\begin{theorem}[Lucchini--Morini] \label{thLM}
Let $L$ be a monolithic primitive group with socle $A$.
For any $t \geq d(L)$ we have
$$ P_{L,A}(t) \> \to \> 1 $$
when $|A| \to \infty$.
\end{theorem}

\subsection{Expander Cayley graphs}

If $G$ is a finite group and $S \subseteq G$, the \emph{Cayley graph} $\Cay(G,S)$ is the undirected
graph with vertex set $G$ and edges $(g, s^{\pm 1} g)$ for $g \in G$ and $s \in S$.
We allow multiple edges, so $\Cay(G,S)$ is always a $d$-regular graph for $d= 2|S|$.
When $N \unlhd G$, with an abuse of notation we write $\Cay(G/N,S)$
to denote the Cayley graph with respect to the projection of $S$ in $G/N$
(note that this projection is a multiset in general).
We write $\gap(G,S)$ to denote the spectral gap of the normalized adjacency matrix of $\Cay(G,S)$.
Thus $\gap(G,S) >0$ if and only if $\Cay(G,S)$ is connected if and only if $G=\langle S \rangle$.
Finally we say that $\Cay(G,S)$ is an \emph{expander} if $\gap(G,S) \geq \delta$
for some constant $\delta >0$ depending only
on the relevant parameters.
For instance, the conclusion of the following theorem is that
$\gap(G,S) \geq \delta(\varepsilon) >0$ with probability $1-o(1)$ as $|G| \to \infty$.

\begin{theorem}[Breuillard--Green--Guralnick--Tao] \label{thBGGT}
Fix $\varepsilon >0$, and let $G$ be an $\varepsilon$-quasirandom simple group.
Then $2$ random elements give an expander Cayley graph with high probability.
\end{theorem}
\begin{proof}
A large finite simple group is $\varepsilon$-quasirandom if and only if
it is a group of Lie type of rank $O(\varepsilon^{-1})$.
Therefore, this is \cite[Theorem~1.2]{BGGT15}.
\end{proof}

\subsection{Golsefidy--Srinivas theorems}

We now report streamlined versions of two essential results of Golsefidy and Srinivas \cite{GS25},
and we refer the reader to \cite{GS25} for slightly more general statements.
We start with the following theorem for direct products,
whose main idea lies in the beautiful paper \cite{GS24}.

\begin{theorem}[Golsefidy--Srinivas] \label{thGSDirPro}
Suppose $G_1$ and $G_2$ are finite groups with trivial center such that
\begin{itemize}
\item $G_1 \times G_2$ is $\varepsilon$-quasirandom, and
\item for each $x \in G_i$ we have $\langle x^{G_i} \rangle = ((x^{\pm 1})^{G_i})^{c}$ for some constant $c \geq 1$.
\end{itemize}
If $G_1 \times G_2 = \langle S \rangle$, then
$$ \gap(G_1 \times G_2,S) \> \gg_{\varepsilon,c,|S|} \> \min \{ \gap(G_1,S) , \gap(G_2,S) \} . $$
\end{theorem}
\begin{proof}
This is \cite[Theorem~A]{GS25}.
It is easy to check that quasirandomness of $G_1 \times G_2$ is equivalent to
a statement about quasirandomness of $G_1$ and $G_2$ and on their mutual cardinalities
(see also Example~\ref{exQ}(a)).
Moreover, we can choose $\alpha_0 = |S|^{-1}$.
\end{proof}

The fact that $S$ is a generating set for the ambient group is critical in Theorem~\ref{thGSDirPro},
and so is in Theorem~\ref{thProjection}.
Essentially, this is needed to avoid phenomena such as
generating the diagonal in the direct product of two isomorphic simple groups.

We come to the second result from \cite{GS25}.
In \cite{LV16}, Lindenstrauss and Varj\'u studied the affine group
$(\mathbb{F}_p)^n \rtimes \SL_n(p)$ for $n$ bounded,
and proved that its expansion can be deduced from expansion in the projection to $\SL_n(p)$.
 The following is a much more general version for quasirandom-by-abelian groups:

\begin{theorem}[Golsefidy--Srinivas] \label{thGSAbe}
Let $G$ be $\varepsilon$-quasirandom and let $A$ be an abelian normal subgroup of $G$ such that,
for every $a \in A$ we have $\langle a^G \rangle = ((a^{\pm 1})^G)^{c}$ for some constant $c \geq 1$.
If $G = \langle S \rangle$, then
$$ \gap(G,S) \> \gg_{\varepsilon,c,|S|} \> \gap(G/A,S) . $$
\end{theorem}
\begin{proof}
This is \cite[Theorem~B]{GS25}.
As observed in \cite[Lemma~17]{GS25}, the group $G$ in the original statement is also quasirandom
(this fact can also be recovered from \cite[Lemma~4.1]{BS25}).
Conversely, if $G$ is $\varepsilon$-quasirandom then $G/A$ is $\varepsilon$-quasirandom
and $|G/A| \geq |G|^\varepsilon$, i.e. $|A| \leq |G/A|^{\varepsilon^{-1} -1}$.
As above, we can choose $\alpha_0 = |S|^{-1}$.
\end{proof}

\begin{remark} \label{remGS}
Theorems~C, D and E in the second half of \cite{GS25} can be deduced from Theorem~\ref{thProjection}.
In fact, the groups in the hypotheses of these theorems are $\varepsilon$-quasirandom for a fixed $\varepsilon>0$
(see \cite[Lemma~20]{GS25}, or \cite[Lemma~4.1]{BS25}).
\end{remark}

\subsection{The solvable radical}

In this subsection we prove that
the solvable radical of an $\varepsilon$-quasirandom group has $\varepsilon$-bounded derived length.
We first need the following two lemmas.

\begin{lemma}[Gr\"un's lemma] \label{lemGrun}
The quotient of a finite perfect group by its center has trivial center.
More precisely, if $G$ is a finite perfect group and $N \subseteq Z(G)$
with $N \unlhd N_0 \lhd G$ and $N_0/N \subseteq Z(G/N)$,
then $N_0 \subseteq Z(G)$.
\end{lemma}
\begin{proof}
See \cite[Lemma~2.10]{BS25}.
\end{proof}

\begin{lemma} \label{lemTwoSteps}
Let $G$ be a finite perfect group and let $R \lhd G$ be a solvable normal subgroup.
Assume that
$$ [R,G]=R' \qquad~\text{and}~\qquad [R',G]=R'' . $$
Then $R$ is abelian.
\end{lemma}
\begin{proof}
We can assume that $R''=1$.
In particular $[R',G]=1$, i.e. $R' \subseteq Z(G)$.
Moreover, $R/R' \subseteq Z(G/R')$.
By Lemma~\ref{lemGrun} we obtain $R \subseteq Z(G)$,
so $R$ is abelian as desired.
\end{proof}

We are ready.

\begin{proposition} \label{propDL} 
If $G$ is $\varepsilon$-quasirandom,
then the derived length of the solvable radical of $G$ is at most $2 \varepsilon^{-1}$.
\end{proposition}
\begin{proof}
Let $R=\Sol(G)$ and let $(R^{(i)})_{i=0}^\ell$ be the derived series of $R=R^{(0)}$,
with $\ell \geq 1$ being the derived length.
Note that $R^{(i+1)} \subseteq [R^{(i)},G]$ for each $i$.
By Corollary~\ref{corQChief}, if
$$ [R^{(i)} , G ] \neq R^{(i+1)} $$
for some $i$, then
$$ |R^{(i)}:R^{(i+1)}| > |G|^\varepsilon . $$
It follows that if $m$ is the number of non-central factors $R^{(i)}/R^{(i+1)}$, then
$$ |G|^{1-\varepsilon} > |R| > |G|^{\varepsilon m} , $$
and so $m < \varepsilon^{-1} -1$.
Now suppose that two consecutive factors are central, namely
$$ [R^{(i)},G] = R^{(i+2)} \qquad \text{and} \qquad [R^{(i+1)},G] = R^{(i+2)} $$
for some $i \in \{0,\ldots,\ell-2\}$.
By Lemma~\ref{lemTwoSteps} we have that $R^{(i)}$ is abelian and so $R^{(i+2)} = R^{(i+1)}=1$,
which is a contradiction.
This implies that the number of central factors is at most $m+1$,
and it follows that $\ell \leq 2m+1 < 2 \varepsilon^{-1} -1$.
\end{proof}

We conclude this section with some examples.

\begin{example} \label{exQ}
The following are notable examples of quasirandom groups:
\begin{itemize}
\item[(a)] Let $G = \prod_{i=1}^n T_i$ be a direct product of perfect groups (for example finite simple groups).
	If $G$ is $\varepsilon$-quasirandom then for all $i$ we have $|T_i| \geq |G|^\varepsilon$ and
	$T_i$ is $\varepsilon$-quasirandom.
	Conversely if $|T_i| \geq |G|^\varepsilon$ and $T_i$ is $\varepsilon$-quasirandom for all $i$,
	then $G$ is $\varepsilon^2$-quasirandom.

\item[(b)] The affine group $G= (\mathbb{F}_p)^2 \rtimes \SL_2(p)$ is $0.19$-quasirandom for large $p$.
This is the group studied in \cite{LV16}
and in fact is a monolithic primitive group with socle $(\mathbb{F}_p)^2$.

\item[(c)] (Big number of generators.)
Let $V=(\mathbb{F}_p)^2$ and for $k \geq 1$ construct $G = V^k \rtimes \SL_2(p)$ with $\SL_2(p)$ acting diagonally.
Then $G$ is a crown-based power and is roughly $(2k+3)^{-1}$-quasirandom for large $p$.
On the other hand, it follows from \cite[Theorem~2.1]{DLRD15} that $d(G) = \lceil k/2 \rceil +1$.
This shows that the linear dependence on $\varepsilon^{-1}$
in Lemma~\ref{lemQGen} and so in Theorem~\ref{thExpRan} is necessary.

\item[(d)] (Big center.)
Let $p$ be an odd prime and let $P$ be the unique non-abelian group of order $p^3$ and exponent $p$.
Consider the semidirect product $G = P \rtimes \SL_2(p)$ with the natural action of $\SL_2(p)$ on $(\mathbb{F}_p)^2$.
Then $G$ is $0.16$-quasirandom for large $p$, but $|Z(G)| = p > |G|^{1/6}$.

\item[(e)] (Big derived length.)
	For each $L \geq 0$ there exist $\varepsilon(L) >0$ and
	infinitely many $\varepsilon$-quasirandom groups with solvable radical of derived length at least $L$.
	The following parallel examples are inspired by the papers \cite{BV12,Bra16}.
	Let $p \geq 5$ be a prime, $n \geq 1$,
	and let $G_{p,n} = \SL_2(\mathbb{Z}/p^n \mathbb{Z})$ or $G_{p,n} = \SL_2(\mathbb{F}_p[x]/(x^n))$.
	Then $G_{p,n}$ is a perfect extension of a solvable group $P_{p,n}$ of order $p^{3(n-1)}$ by $\SL_2(p)$.
	In particular, $G_{p,n}$ is $(3n)^{-1}$-quasirandom for all large $p$,
	and it can be checked that the derived length of $P_{p,n}$ grows with $n$ uniformly in $p$.

\item[(f)] The following split extensions are closely related
to the perfect semidirect products considered in \cite[Page~2367]{GS25}
(see also \cite[Example~4.7]{BS25}, where there is a minor mistake).
	Let $p \geq 5$ be a prime and $d \geq 2$.
	Let $H_{d,p}$ be the group of the $2d \times 2d$ upper triangular matrices whose diagonal entries
	are identical matrices in $\SL_2(p)$,
	and the upper entries are arbitrary $2 \times 2$ matrices over $\mathbb{F}_p$.
	For instance,
	$$
	H_{3,p} \> = \>
	\left \{ 
	\left(
	\begin{array}{ccc}
		g & * & *  \\
		0 & g & *  \\
		0 & 0 & g  \\
	\end{array}
	\right) : g \in \SLL(p) , \> * \in \mathrm{M}_2(\mathbb{F}_p)
	\right \} .
	$$
	If $U_{d,p}$ is the subgroup where $g =1$, then $|U_{d,p}|=p^{2d(d-1)}$ and $H_{d,p} = U_{d,p} \rtimes \SL_2(p)$.
	Now $H_{d,p}$ is not perfect, but its perfect core has shape $\widetilde{U}_{d,p} \rtimes \SL_2(p)$
	for some invariant subgroup $\widetilde{U}_{d,p}$,
	and is $(2d^2)^{-1}$-quasirandom for all large $p$.
\end{itemize}
\end{example}

\vspace{0.1cm}
\section{Generation of quasirandom groups} \label{secGen}

The purpose of this section is to prove
the following weak version of Theorem~\ref{thExpRan}.

\begin{proposition} \label{propRanGen} 
If $G$ is $\varepsilon$-quasirandom,
then $1 + \lfloor \varepsilon^{-1} \rfloor$ random elements generate $G$ with high probability.
\end{proposition} 

We first obtain an even weaker, deterministic version.
The next proof does not use crown theory
but implements the multi-step process that we have described in the introduction,
which will be used again in Section~\ref{secExp}.
We recall that $d(G)$ is the minimum number of generators.

\begin{lemma} \label{lemQGen}
If $G$ is $\varepsilon$-quasirandom and sufficiently large with respect to $\varepsilon$,
then $d(G) \leq 1 + \varepsilon^{-1}$.
\end{lemma}
\begin{proof}
We first show that $d(G/\Sol(G)) =2$.
By Theorem~\ref{thBS} we can assume that $G/\Sol(G) = \prod_{i=1}^n T_i^{\alpha_i}$
is a direct product of boundedly many finite simple groups.
Now it is easy to see that $d(G/\Sol(G)) = \max_i d(T_i^{\alpha_i})$.
Since $\alpha_i \ll_\varepsilon 1$ and $|T_i| \geq |G|^\varepsilon$ for each $i$,
the claim follows from \cite[Section~4]{Wie74}.

We start a process in which at each step we quotient the group by its center
and then by a non-central minimal abelian normal subgroup
(observe that this is always possible by Gr\"un's lemma).
By Remark~\ref{remZF} we have $d(G)=d(G/Z(G))$,
and it is easy to see that $d(G) \leq d(G/A) +1$ if $A$ is a minimal normal subgroup of $G$.
By Corollary~\ref{corQChief},
at each step of the process we are dividing the order of the solvable radical by more than $|G|^\varepsilon$,
and since $|\Sol(G)| < |G|^{1-\varepsilon}$ by quasirandomness,
we terminate in at most $\varepsilon^{-1}-1$ steps.
The proof follows.
\end{proof}

To prove a probabilistic version, we need two more results.
 We first describe the crown-based powers that are also quasirandom groups.

\begin{lemma} \label{lemQCrown}
Let $L$ be a monolithic primitive group with socle $A$,
and let $L_k$ be a crown-based power.
Suppose that $L_k$ is $\varepsilon$-quasirandom and sufficiently large with respect to $\varepsilon$.
If $A$ is non-abelian, then $L=A$ is a finite simple group and $L_k=A^k$.
Moreover, in any case we have
$$ |A| > |L_k|^\varepsilon \hspace{0.5cm} \text{ and } \hspace{0.5cm} k < \varepsilon^{-1} . $$
\end{lemma}
\begin{proof}
In general, note that $L_k$ has a non-central normal subgroup isomorphic to $A$,
and so $|A| > |L_k|^\varepsilon$ by Corollary~\ref{corQChief}.
Since $L_k$ also contains a copy of $A^k$, we have $|L_k|^{k \varepsilon} < |A|^k \leq |L_k|$,
which gives $k < \varepsilon^{-1}$.

If $A$ is not abelian, then $\Sol(L)=1$ and $L$ is the direct product of simple groups from Theorem~\ref{thBS}.
Since $L$ is monolithic, it is actually simple.
\end{proof} 

Next, we deal with the parameters $\gamma_A$ and $q_A$ in the definition of $\widetilde{P}_{L,i}(t)$.

\begin{lemma} \label{lemComp}
Let $L$ be a perfect monolithic primitive group with abelian socle $A$.
Then $|\End_L(A)| \leq |A|^{1/2}$, and $|C_{\Aut(L)}(L/A)| < |A|^2$.
\end{lemma}
\begin{proof}
 We have that $A$ is elementary abelian of order $p^n$ for some prime $p$ and $n \geq 1$.
Thus we may regard $A$ as a vector space over $\mathbb{F}_p$ of dimension $n$.
The conjugation action of $L$ on $A$ induces a faithful irreducible action of
$L/A$ on $A$, and $A$ is an irreducible $\mathbb{F}_p[L/A]$-module.
Now the centralizer of $L/A$ in $\Aut(A)$ consists precisely of the invertible
$L$-endomorphisms of $A$, namely $C_{\Aut(A)}(L/A) = \End_L(A)^\times$.

By Schur's lemma and Wedderburn's little theorem, $\End_L(A)$ is a finite field,
say $\End_L(A) \cong \mathbb{F}_{p^r}$ for some $r \geq 1$.
Therefore $\End_L(A)^\times \cong \mathbb{F}_{p^r}^\times$.
As $A$ is a $\mathbb{F}_{p^r}$-vector space, if $m=\dim_{\mathbb{F}_{p^r}}(A)$, then $n=mr$.
If $m=1$, then $\Aut_{\mathbb{F}_{p^r}}(A)$ is abelian and so is $L/A$,
against the assumption that $L$ is perfect.
Thus $m \geq 2$ and so $p^r \leq p^{n/2} = |A|^{1/2}$.

Finally we have $|C_{\Aut(L)}(L/A)| = |A| |\End_L(A)^\times| |H^1(L/A,A)|$.
By \cite[Theorem~1]{GH98} we deduce $|H^1(L/A,A)| \leq |A|^{1/2}$ and the proof follows.
\end{proof}

We are ready to implement Detomi--Lucchini theorem.

\begin{proof}[Proof of Proposition~\ref{propRanGen}] 
Let $G$ be $\varepsilon$-quasirandom and let $t= 1+ \lfloor \varepsilon^{-1} \rfloor$.
By Theorem~\ref{thDL}, our goal is to prove that 
$$ \prod_{A} \left( \prod_{1 \leq i \leq \delta_A(G)} \widetilde{P}_{L_A,i}(t) \right) 
\> \to \> 1 $$
when $|G| \to \infty$,
where $A$ runs in the set of irreducible $G$-groups $G$-equivalent to a non-Frattini
chief factor of $G$,
and $L_A$ is the monolithic primitive group associated with $A$.

By Corollary~\ref{corQChief}, for each non-Frattini chief factor $A$ we have $|A| > |G|^\varepsilon$.
In particular, there are at most $\varepsilon^{-1}$ such factors.
Therefore, it is sufficient to prove that for each $A$ we have
$$ \prod_{1 \leq i \leq \delta_A(G)} \widetilde{P}_{L_A,i}(t)  \> \to \> 1 $$
when $|A| \to \infty$.
For the associated monolithic primitive group $L=L_A$,
the crown-based power $L_{\delta_A(G)}$ appears as a nontrivial quotient of $G$.
It follows that $L_{\delta_A(G)}$ is $\varepsilon$-quasirandom,
so $\delta_A(G) \leq \varepsilon^{-1}$ by Lemma~\ref{lemQCrown}.
Therefore it is enough to prove that for all $A$ and $i \leq \varepsilon^{-1}$
we have $\widetilde{P}_{L_A,i}(t)  \to 1$ uniformly.
Since $t \geq d(L_A)$ by Lemma~\ref{lemQGen},
by the definition of $\widetilde{P}_{L_A,i}(t)$ and Theorem~\ref{thLM},
it is actually sufficient to show that, for all $A$ we have
$$ \frac{\gamma_A \, (q_A)^{1/\varepsilon}}{|A|^{1+ \lfloor \varepsilon^{-1} \rfloor}} \> \to \> 0 $$
when $|A| \to \infty$,
where $\gamma_A = |C_{\Aut(L)}(L/A)|$ and $q_A = |\End_L(A)|$ if $A$ is abelian and $q_A =1$ otherwise.

If $A$ is abelian, then by Lemma~\ref{lemComp} we have
$\gamma_A \, (q_A)^{1/\varepsilon} < |A|^{2+\varepsilon^{-1}/2}$
and we are done because we can assume $\varepsilon < 1/3$,
and so $\lfloor \varepsilon^{-1} \rfloor - \varepsilon^{-1}/2 >1$.
Otherwise $L=A$ by Lemma~\ref{lemQCrown} and we have $\gamma_A =|\Aut(A)|$ where $A$ is a finite simple group.
In this case it is well known that $|\Aut(A)| \leq |A|^{d(A)} = |A|^2$ and the proof is complete.
\end{proof}

\vspace{0.1cm}
\section{Diameter of quasirandom groups}  \label{secDia}

Obtaining an upper bound for the worst-case diameter of a finite group is in general a difficult problem.
A fruitful strategy first devised by Helfgott \cite{Hel08} to get polylogarithmic bounds consists
in proving a ``growth'' theorem of the following type:
\emph{If $\SL_2(p) = \langle A \rangle$,
then either $|A^3| \geq |A|^{1+\delta}$ for some fixed $\delta>0$ or $A^3=\SL_2(p)$.}
Unfortunately, a statement of this type fails already in $\SL_2(p) \times \SL_2(p)$,
as $A$ might have the form $A_1 \times A_2$ with $A_1$ small and $A_2$ large in $\SL_2(p)$.
In the recent paper \cite{EMST26},
the authors obtain a general-purpose polylog-type diameter bound for arbitrary finite groups.
That bound depends on the maximal exponent of a normal abelian section.
As stressed in Example~\ref{exQ}(b-d),
there exist quasirandom groups with problematic abelian normal subgroups,
or even large central cyclic subgroups,
so a different method is required to prove Theorem~\ref{thDiam}.

If $G$ is a finite group and $S \subseteq G$, we write $\ell_S$ for the length function with respect to $S$,
i.e. if $g \in G$ then $\ell_S(g)$ is the length of the minimal representation of $g$ as a product of elements in $S \cup S^{-1}$.
    For a subset $X \subseteq G$ we write $\ell_S(X) = \max_{x \in X} \ell_S(x)$.
    Therefore $\diam(G, S) = \ell_S(G)$, and
    $$ \diam(G) = \max \{\diam(G, S) : S~\text{a generating set for}~G\} . $$
    
    The following tools are basic.

\begin{lemma}[Basic diametry] \label{lemDiamTools}
Let $G$ be a finite group. Then
\begin{itemize}
\item[(i)] For $S, X, Y \subseteq G$ we have
   $$ \ell_S(Y) \> \leq \> \ell_S(X) \ell_X(Y) . $$
    
    \item[(ii)] Suppose $K \unlhd L \unlhd H \leqslant G$. Then
    $$ \ell_S(H/K) \> \leq \> \ell_S(H/L) + \ell_S(L/K) . $$
    
    \item[(iii)] If $G = \langle S \rangle$ and $N \unlhd G$,
   then there exists $X \subseteq N$ such that $N = \langle X \rangle$ and $\ell_S(X) \leq 2 \diam(G/N) + 1$.
\end{itemize}
\end{lemma}
\begin{proof}
See \cite[Section~2]{EMST26}.
\end{proof}

The next growth-type result is due to Eberhard--Murphy--Pyber--Szab\'o \cite{EMPS25}.
In their paper they use a different notation where a group is called $K$-quasirandom
if every nontrivial irreducible representation has degree at least $K$.
Using the results in Section~\ref{secGen},
we recover the following version without any dependence on the size of a generating set.

\begin{lemma}[Affine conjugating trick] \label{lemACT}
Let $G$ be an $\varepsilon$-quasirandom group acting on an abelian group $A$.
If $X \subseteq A$ is a $G$-invariant symmetric generating set for $A$,
then
$$ \frac{|X^3|}{|X|} > |G|^{\varepsilon/21} \qquad \text{or} \qquad [A,G] \subseteq X^{14/\varepsilon} . $$
\end{lemma}
\begin{proof}
See \cite[Lemma~1.8]{EMPS25},
and use that $G$ can be generated by at most $1 + \varepsilon^{-1} \leq 2 \varepsilon^{-1}$ elements
by Lemma~\ref{lemQGen}.
\end{proof}

Lemma~\ref{lemACT} is empty if $G$ acts trivially on $A$,
and in fact we need a different argument to deal with central elements.
The proof of the following nice fact uses basic character theory.

\begin{lemma}[Commutator width] \label{lemWid}
If $G$ is $\varepsilon$-quasirandom,
then every element of $G$ is the product of at most $\varepsilon^{-1}$ commutators.
\end{lemma}
\begin{proof}
Fix $g \in G$.
A classical result of Frobenius says that $g$ is a commutator if and only if
$$ \sum_{\chi \in \Irr(G)} \frac{\chi(g)}{\chi(1)} > 0 $$
where $\chi$ runs among the irreducible characters of $G$.
More generally, by \cite[Lemma~9.1]{Sha09},
$g$ can be written as the product of $k \geq 1$ commutators if and only if
$$ \sum_{\chi \in \Irr(G)} \frac{\chi(g)}{\chi(1)^{2k-1}} > 0 . $$
Therefore, because $|\chi(g)| \leq \chi(1)$ for all $\chi \in \Irr(G)$,
for the thesis to hold it is sufficient to show that
$$ \sum_{\chi \in \Irr(G) \setminus 1} \chi(1)^{2-2k} < 1 $$
for some $k \leq \varepsilon^{-1}$.
Since there are less than $|G|$ irreducible characters, for $k \geq 2$ the $\varepsilon$-quasirandomness gives
$$ \sum_{\chi \in \Irr(G) \setminus 1} \chi(1)^{2-2k} \leq 
\sum_{\chi \in \Irr(G) \setminus 1} |G|^{-2\varepsilon (k-1)} < 
|G|^{1 -2\varepsilon (k-1)} . $$
If $k = \lfloor \varepsilon^{-1} \rfloor$ then $k-1 \geq (2 \varepsilon)^{-1}$ and we are done
(recall that $\varepsilon < 1/2$).
\end{proof}

It is well known that a much stronger result holds for the finite simple groups:
by the famous Ore's conjecture, every element is a commutator \cite{LOST10}.
Anyway, it is notable that the proof of Lemma~\ref{lemWid} is elementary.

We are ready to prove our diameter bound for quasirandom groups.
We proceed along the derived series of the solvable radical,
taking advantage of Proposition~\ref{propDL}.

\begin{proof}[Proof of Theorem~\ref{thDiam}] 
Let $R=\Sol(G)$, and let $\ell \geq 0$ be the derived length of $R$.
By Proposition~\ref{propDL} we have $\ell \leq 2 \varepsilon^{-1}$, and we proceed by induction on $\ell$.
If $\ell=0$, then $R$ is trivial and by Theorem~\ref{thBS} we have that
$G$ is the direct product of boundedly many finite simple groups of Lie type of bounded rank,
say $G= \prod_{i=1}^n T_i$ with $n \ll_\varepsilon 1$.
By iterating Lemma~\ref{lemDiamTools}(iii) and \cite{BGT11,PS16} we have
$$ \diam(G) \leq \prod_{i=1}^n (2 \diam(T_i) +1) \leq (\log|G|)^{O_\varepsilon(1)} . $$
(A better estimate for $\diam(\prod_{i=1}^n T_i)$ was proven in \cite{Don22},
but we do not need that result.)

Now suppose $\ell \geq 1$ and let $A = R^{(\ell-1)} \lhd G$.
Note that $A$ is abelian and $G/A$ has derived length $\ell -1$.
By induction we have $\diam(G/A) \leq (\log|G|)^{O_\varepsilon(1)}$.
If $G = \langle S \rangle$ (we can assume that $S$ is symmetric),
by Lemma~\ref{lemDiamTools}(ii) we have
$$ \diam(G,S) \leq \ell_S([A,G]) + \ell_S(G/[A,G]) . $$
By Lemma~\ref{lemDiamTools}(iii)
we obtain $A=\langle X \rangle$ where $\ell_S(X) \leq (\log|G|)^{O_\varepsilon(1)}$.
Since $A$ is abelian and $\ell_S(G/A) \leq (\log|G|)^{O_\varepsilon(1)}$,
the normal closure $X^G$ still satisfies $\ell_S(X^G) \leq (\log|G|)^{O_\varepsilon(1)}$
and clearly is a $G$-invariant symmetric generating set for $A$.
Applying Lemma~\ref{lemACT} iteratively with $X^G,(X^G)^3,\ldots,(X^G)^{3^i}$,
we obtain $\ell_{X^G}([A,G]) \leq (\log|[A,G]|)^{O_\varepsilon(1)}$.
By Lemma~\ref{lemDiamTools}(i) we have
$$ \ell_S([A,G]) \leq \ell_S(X^G) \cdot \ell_{X^G}([A,G]) \leq (\log|G|)^{O_\varepsilon(1)} . $$

Let $\overline{G}=G/[A,G]$ and $\overline{A}=A/[A,G]$, so that it remains to bound $\ell_S(\overline{G})$.
Observe that $\overline{A} \subseteq Z(\overline{G})$.
By Lemma~\ref{lemWid}, every element of $\overline{G}$ is the product of at most $\varepsilon^{-1}$ commutators.
Moreover, the commutator map $\overline{G} \times \overline{G} \to \overline{G}$ has the same image as the
induced map $G/A \times G/A \to \overline{G}$.
Up to replacing the constant $O_\varepsilon(1)$, we obtain
$$ \ell_S(\overline{G}) \leq 4 \varepsilon^{-1} \ell_S(G/A) \leq (\log|G|)^{O_\varepsilon(1)} $$
as desired.
\end{proof} 

\begin{remark}[Dependence on $\varepsilon$]
The exponent in Theorem~\ref{thDiam} has to depend on $\varepsilon$:
just take any perfect group with large diameter,
for example the deleted permutation module $(\mathbb{F}_p)^4 \rtimes \Alt(5)$ for large $p$.
The $O_\varepsilon(1)$ term can be computed explicitly,
but it obviously has to depend on what we are able to say when $G$ is simple.
In this case $G=T$ is a group of Lie type of rank $r=O(\varepsilon^{-1})$.
The current best known general bounds of the form $\diam(T) \leq (\log|T|)^{C(r)}$
are proven in \cite{HMPQ19}, but are considered far from the truth
(in particular the famous Babai's conjecture \cite{BS92} predicts that $C(r)$ can be chosen to be a constant).
\end{remark}

\vspace{0.1cm} 
\section{Expansion in quasirandom groups} \label{secExp}

In this section we prove Theorem~\ref{thProjection} and its consequences.

\subsection{Central extensions} 

Let $G$ be a finite perfect group and $S \subseteq G$.
It is a consequence of Remark~\ref{remZF}
that $\gap(G,S) >0$ if and only if $\gap(G/Z(G),S) >0$,
and we seek a quantitative strengthening of this fact.
For this purpose it is convenient to deviate from spectral gaps and use the \emph{Kazhdan constant}
(see \cite{BHV08} for an introduction to this parameter and its applications).
This is defined by
$$ \Kaz(G,S) \> := \> \inf_{\substack{\pi \colon G \to U(V) \\ V^G =0}} 
\inf_{\| {\bf v} \| =1} \max_{s \in S}  \| \pi(s) {\bf v} - {\bf v} \|^2 , $$
where the first infimum runs over all unitary representations $\pi \colon G \to U(V)$ without
non-zero $G$-invariant vectors, and the second infimum runs over all unit vectors ${\bf v} \in V$
(note that we consider squared norms for convenience).
The following inequalities are standard:
$$ \frac{\Kaz(G,S)}{2|S|} \> \leq \> \gap(G,S) \> \leq \> \frac{\Kaz(G,S)}{2} , $$
see \cite[Proposition~III(2)]{HRV93}.
Often it is also technically useful to work with a variation involving only irreducible representations:
$$ \Kazz(G,S) \> := \> \min_{\substack{\pi \colon G \to U(V) \\ \pi \neq 1, \text{ \tiny{irreducible}}}} 
\inf_{\| {\bf v} \| =1} \max_{s \in S}  \| \pi(s) {\bf v} - {\bf v} \|^2 . $$
It follows from \cite[Proposition~III(1)]{HRV93} that
$$ \Kaz(G,S) \> \leq \> \Kazz(G,S) \> \leq \> |S| \Kaz(G,S) . $$
(Decomposing $(\pi,{\bf v})$, the right-hand inequality can be obtained
by collecting the irreducible representations giving the same $s \in S$ in the definition of $\Kazz(G,S)$.)
The next result provides the desired improvement of Remark~\ref{remZF}.
We were inspired by the proof of \cite[Lemma~1.7.10]{BHV08},
which the authors attribute to Serre.

	\begin{lemma} \label{lemKazZ}
       Let $G$ be a perfect group and $S \subseteq G$.
       Then
       $$ \Kazz(G,S) \> \geq \> \tfrac{1}{4} \Kaz(G/Z(G),S) . $$
       \end{lemma}
       \begin{proof}
       We have to show that for every nontrivial irreducible representation
       $\pi \colon G \to U(V)$ and unit vector ${\bf v} \in V$
       there is $x \in S$ such that
       $$ \| \pi(x) {\bf v} - {\bf v} \|^2 \geq \tfrac{1}{4} \Kaz(G/Z(G),S) . $$
       Let $\pi \colon G \to U(V)$ be a nontrivial irreducible representation of degree $n$.
       Since $G$ is perfect, we have $n \geq 2$.
       From Schur's lemma, the elements of $Z(G)$ act as scalars, so the tensor product
       $\pi \otimes \overline{\pi} \colon G \to U(V \otimes \overline{V})$
       factorizes as a $G/Z(G)$ representation of degree $n^2$.
       For every $x \in G$ and unit vector ${\bf v} \in V$, we have
      $$
       \| \pi(x) {\bf v} - {\bf v} \|^2 \> \geq \>
       \tfrac{1}{2} \| (\pi \otimes \overline{\pi})(x) ({\bf v} \otimes {\bf \overline{v}}) - ({\bf v} \otimes {\bf \overline{v}}) \|^2 .
       $$
       (We refer to \cite[Lemma~1.7.10]{BHV08} for more explanation on this inequality.)
       
       One issue is that $\pi \otimes \overline{\pi}$ does have nontrivial invariant vectors.
       However, because $\pi$ is irreducible,
       the multiplicity of the trivial representation in $\pi \otimes \overline{\pi}$ is just $1$.
       Let $\mathcal{Z} \subset V$ be the invariant line,
       and let $\mathcal{Z}^\perp$ be its orthogonal subspace.
       Let us write ${\bf v} \otimes {\bf \overline{v}} = {\bf w}_1 + {\bf w}_2$,
       with ${\bf w}_1 \in \mathcal{Z}$ and ${\bf w}_2 \in \mathcal{Z}^\perp$.
       By the invariance of ${\bf w}_1$ we have
       $$ \| (\pi \otimes \overline{\pi})(x) ({\bf v} \otimes {\bf \overline{v}}) - ({\bf v} \otimes {\bf \overline{v}}) \|^2  =
       \| (\pi \otimes \overline{\pi})_{\mathcal{Z}^\perp}(x) {\bf w}_2 - {\bf w}_2 \|^2 . $$
       Now $(\pi \otimes \overline{\pi})_{\mathcal{Z}^\perp}$ has no nontrivial invariant vectors,
      and it remains to estimate $\| {\bf w}_{_2} \|$.
     Since ${\bf w}_1$ and ${\bf w}_2$ are orthogonal we have
      $1 = \| {\bf v} \otimes {\bf \overline{v}} \|^2 = 
      \| {\bf w}_1 \|^2 + \| {\bf w}_2 \|^2$.
      Fix a basis ${\bf e}_1,\ldots,{\bf e}_n$ for $V$,
      so that $\mathcal{Z}$ is generated by $\sum_i {\bf e}_i \otimes \overline{{\bf e}_i}$.
       Write ${\bf v}=\sum_i a_i {\bf e}_i$, so that
       $$ {\bf v} \otimes {\bf \overline{v}} = \sum_{i,j} (a_i a_j) \> {\bf e}_i \otimes \overline{{\bf e}_j} . $$
       Since ${\bf w}_1$ is the projection of ${\bf v} \otimes {\bf \overline{v}}$ on $\mathcal{Z}$,
       and $\| \sum_{i} {\bf e}_i \otimes \overline{{\bf e}_i} \| = \sqrt{n}$,
       we have
       \[\begin{split}
		\sqrt{n} \cdot \| {\bf w}_1 \|
		&= \langle {\bf v} \otimes {\bf \overline{v}} , \sum_{i} {\bf e}_i \otimes \overline{{\bf e}_i} \rangle
		\\&= \sum_i \langle \sum_{h,k} a_ha_k ({\bf e}_k \otimes \overline{{\bf e}_k}) , {\bf e}_i \otimes \overline{{\bf e}_i} \rangle
		\\&= \sum_i \langle (a_i)^2 ({\bf e}_i \otimes \overline{{\bf e}_i}) , {\bf e}_i \otimes \overline{{\bf e}_i} \rangle
		\\&= \sum_i (a_i)^2 = \| {\bf v} \|^2 = 1 .
	\end{split} \, \] 
       It follows that $\| {\bf w}_1 \|^2 = 1/n \leq 1/2$, and then $\| {\bf w}_2 \|^2 \geq 1/2$.
       
      Finally, by definition of $\Kaz(G/Z(G),S)$ there exists $x \in S$ such that
      $$ \| (\pi \otimes \overline{\pi})_{\mathcal{Z}^\perp} (x) {\bf w}_2 - {\bf w}_2 \|^2 \geq
       \tfrac{1}{2} \Kaz(G/Z(G),S) , $$
    and the proof is complete.
       \end{proof}

We can easily express Lemma~\ref{lemKazZ} in terms of spectral gaps.

\begin{corollary} \label{corExpC}
Let $G$ be a finite perfect group and $S \subseteq G$.
Then
$$ \gap(G,S) \> \gg_{|S|} \> \gap(G/Z(G),S) . $$
\end{corollary}
\begin{proof}
By the previous discussion we have
$$ \gap(G,S) \geq \frac{\Kaz(G,S)}{2|S|} \geq \frac{\Kazz(G,S)}{2|S|^2} . $$
Hence, by Lemma~\ref{lemKazZ},
$$ \Kazz(G,S) \geq \frac{\Kaz(G/Z(G),S)}{4} \geq \frac{\gap(G/Z(G),S)}{8} . \qedhere $$
\end{proof}

\subsection{Proof of Theorems~\ref{thProjection} and~\ref{thExpRan}} 

We can now put the ingredients together.

\begin{proof}[Proof of Theorem~\ref{thProjection}]
Let $G = \langle S \rangle$.
We will first prove that
\begin{equation} \label{eqProj}
\gap(G,S) \> \gg_{\varepsilon,|S|} \> \gap(G/\Sol(G),S) . 
\end{equation} 
We implement the same process as in the proof of Lemma~\ref{lemQGen},
where at each step we quotient the group by its center
and then by a non-central minimal abelian normal subgroup.
As before, we terminate in less than $\varepsilon^{-1}$ steps.
First, we can handle central extensions with Corollary~\ref{corExpC}.
Replace $G$ with $G/Z(G)$ and let $A \lhd G$ be a minimal abelian normal subgroup.
By minimality we have $[A,G]=A$ and $\langle a^{G} \rangle =A$ for each $a \in A \setminus 1$.
By applying Lemma~\ref{lemACT} iteratively, with $X=(a^{\pm 1})^{G}$, we get
$$ \langle a^{G} \rangle = ((a^{\pm 1})^{G})^{O_\varepsilon(1)} . $$
Thus we can use Theorem~\ref{thGSAbe} and the proof of~(\ref{eqProj}) is complete.

Now we can replace $G$ with $G/\Sol(G)$.
We can assume that $G$ is sufficiently large with respect to $\varepsilon$, so
by Theorem~\ref{thBS}, $G = \prod_{i=1}^n T_i$
for some $T_i$'s finite simple groups of Lie type of $\varepsilon$-bounded rank and $n \ll_\varepsilon 1$.
Let $P_k = \prod_{i=1}^{k} T_i$.
For each $1 \leq k \leq n-1$, apply Theorem~\ref{thGSDirPro} with $G_1=P_k$ and $G_2=T_{k+1}$.
For each $T_i$, the necessary condition on conjugacy classes is guaranteed by \cite[Theorem~1]{LL98},
and the case of a direct product follows easily.
\end{proof}

The fact that $G$ is $\varepsilon$-quasirandom is crucial in Theorem~\ref{thProjection}.
In fact, if $G$ is any group with a unique maximal normal subgroup of bounded index,
then it is easy to make the unique simple quotient expand without giving significant information on $G$.
Easy examples are the cyclic group $C_{2^n}$ and the deleted permutation module
$(\mathbb{F}_p)^4 \rtimes \Alt(5)$ for large $p$.
We now give an example of a perfect group $G=V \rtimes H$ with $G = \langle S \rangle$,
where $\Cay(H,S)$ is an expander, $\log|V|=o(\log|H|)$ and $H$ acts irreducibly on $V$,
but $\Cay(G,S)$ is not an expander.

\begin{example} \label{exBad}
Let $H=\Alt(n)$ for $n$ odd and let $G = V \rtimes H$ be the deleted permutation module
with $V \cong (\mathbb{F}_2)^{n-1}$, where $V$ is naturally embedded in $(\mathbb{F}_2)^n$.
By Kassabov \cite{Kas07} there is a bounded-size set $T \subseteq H$ such that $\Cay(H,T)$ is an expander.
However, if ${\bf v} = (1,1,0,\ldots,0) \in V$ and
$$ S = \{ {\bf v}, T \} \subseteq G , $$
then $\Cay(G,S)$ is not an expander.
To see this, the easiest way is to note that the \emph{vertex expansion ratio} of $\Cay(G,S)$ goes to zero
(see \cite[Definition~1.1]{Lub12}).
In fact, the subset
 $$ A = \{ \left( {\bf a} , \sigma \right) : {\bf a}_{\sigma^{-1}(1)}=0 \} \subseteq G $$
 has size $|G|/2$, is invariant under multiplication by $H$,
 and almost invariant under multiplication by ${\bf v}$.
For $({\bf a},\sigma) \in A$ and $\tau \in H$, we have
$$ \tau \cdot ({\bf a},\sigma) = (  {\bf a}^\tau , \tau \sigma ) \in A . $$
Moreover, we have
$$ {\bf v} \cdot ({\bf a},\sigma)=({\bf a}+{\bf v},\sigma), $$
so ${\bf v} \cdot ({\bf a},\sigma) \notin A$ if and only if
$$ ({\bf a}+{\bf v})_{\sigma^{-1}(1)}=1. $$
Since ${\bf a}_{\sigma^{-1}(1)}=0$, this is equivalent to ${\bf v}_{\sigma^{-1}(1)}=1$.
But ${\bf v}_i=1$ if and only if $i\in\{1,2\}$, so
$$
{\bf v}\cdot ({\bf a},\sigma)\notin A
\quad \Longleftrightarrow \quad
\sigma^{-1}(1)\in \{1,2\}.
$$
It follows that
$$
\frac{|S A \setminus A|}{|A|} =
\frac{|{\bf v} A \setminus A|}{|A|}
\leq \frac{2}{n} \to 0
$$
as $n \to \infty$.
\end{example}

To state our next result we need some notation concerning random walks in a Cayley graph $\Cay(G,S)$.
For each $g \in G$ and $n \geq 1$, we write
$$ \mu_{G,S}^{(n)}(g) \> := \> \Prob_{s_1,\ldots,s_n \in S \cup S^{-1}}( s_1 \cdots s_n = g ) . $$
It is easy to see that $\mu_{G,S}^{(2n)}(1) = \| \mu_{G,S}^{(n)} \|^2$,
and an application of the Cauchy--Schwarz inequality gives
\begin{equation} \label{eqRW}
\mu_{G,S}^{(2n)}(g) \> \leq \> \mu_{G,S}^{(2n)}(1) 
\end{equation}
for all $g \in G$ and $n \geq 1$.
This implies that $\mu_{G,S}^{(2n)}(1) \geq |G|^{-1}$ for all $n$.
Moreover,
it is well known that $\mu_{G,S}^{(2n)}(1) \to |G|^{-1}$ if $\Cay(G,S)$ is connected and not bipartite.
For $X \subseteq G$, let $\mu_{G,S}^{(n)}(X) = \sum_{x \in X} \mu_{G,S}^{(n)}(x)$.

As a consequence of Theorem~\ref{thProjection},
we are able to remove the product theorem (Stage 2)
from the hypotheses in the Bourgain--Gamburd expansion machine
(see \cite[Theorem~1.4.2]{Tao15}).
Since non-concentration in proper subgroups is necessary for expansion,
we actually obtain the following result:

\begin{corollary}[Bourgain--Gamburd equivalence] \label{corBGFinal}
Let $G$ be $\varepsilon$-quasirandom for some fixed $\varepsilon >0$ and let $S \subseteq G$.
Then $\Cay(G,S)$ is an expander \underline{if and only if} there exists $\tau>0$ such that
$$ \sup_{H<G} \mu_{G,S}^{(2n)}(H) \> \leq \> |G|^{-\tau} $$
for some $n \leq \tau^{-1} \log|G|$.
\end{corollary}
\begin{proof}
Since $\mu_{G,S}^{(2n)}(\langle S \rangle) =1$ for all $n$,
non-concentration in proper subgroups obviously implies that $S$ is a generating set.
Let $N \lhd G$ be a maximal normal subgroup
and observe that $G/N$ is an $\varepsilon$-quasirandom simple group,
thus a simple group of Lie type of rank $O(\varepsilon^{-1})$.
Considering the random walk on $\Cay(G/N,S)$, for $N<H<G$ we have
$$ \mu_{G/N,S}^{(2n)}(H/N) = \mu_{G,S}^{(2n)}(H) \leq |G|^{-\tau} \leq |G/N|^{-\tau} $$
for some $n \leq \tau^{-1} \log|G|$ .
Since $|G/N| \geq |G|^\varepsilon$,
we have $\log|G| \leq \varepsilon^{-1} \log|G/N|$
and non-concentration holds in $G/N$.
By the classical Bourgain--Gamburd machine \cite[Theorem~1.4.2]{Tao15} and the product theorem \cite{BGT11,PS16},
we obtain that $\Cay(G/N,S)$ is a $\delta$-expander for some $\delta >0$ depending only on $\varepsilon$ and $\tau$.
Because $N$ is arbitrary, the claim follows by Theorem~\ref{thProjection}.

Conversely, if $\gap(G,S) \geq \delta$ for some $\delta>0$,
then $\mu_{G,S}^{(2n)}(1) \leq |G|^{-1} + |G|^{-10}$ for some $n \ll_\delta \log|G|$
(see \cite[Exercise~1.1.16]{Tao15}).
Thus it is easy to see that $\mu_{G,S}^{(2n)}(1) \leq |G|^{-1 +\delta'}$ holds for $n \ll_{\delta,\delta'} \log|G|$,
where $\delta'>0$ can be made arbitrarily small.
Since $|H| < |G|^{1-\varepsilon}$ for every proper subgroup $H$ by quasirandomness,
combining with~(\ref{eqRW}) we obtain
$$ \mu_{G,S}^{(2n)}(H) \leq |H| \, \mu_{G,S}^{(2n)}(1) < |G|^{\delta' - \varepsilon} $$
for some $n \ll_{\delta,\delta'} \log|G|$.
Hence non-concentration is satisfied by choosing $\delta' < \varepsilon$.
\end{proof}

To conclude, we prove that
$1 + \lfloor \varepsilon^{-1} \rfloor$ random elements give an expander with high probability.

\begin{proof}[Proof of Theorem~\ref{thExpRan}]
Let $S \subseteq G$ be a subset of $1 + \lfloor \varepsilon^{-1} \rfloor$ random elements.
Then with high probability $G=\langle S \rangle$ by Proposition~\ref{propRanGen}.
The proof follows by Theorems~\ref{thProjection} and~\ref{thBGGT},
since random elements in $G$ project to random elements in each quotient.
\end{proof}


\vspace{0.1cm} \noindent
{\bfseries Acknowledgments:}
I am indebted to Sean Eberhard,
in particular for making me realize the validity of Lemma~\ref{lemWid}.
I am also grateful to Andrea Lucchini and Laci Pyber for helpful conversations.

   \vspace{0.2cm}

\end{document}